\documentclass[a4paper,reqno]{amsart}

\usepackage{dsfont,tikz-cd,hyperref}

\newtheorem{theorem}{Theorem}[section]

\newtheorem{lemma}[theorem]{Lemma}
\newtheorem{proposition}[theorem]{Proposition}

\DeclareMathOperator{\add}{add}
\DeclareMathOperator{\Cok}{Cok}
\DeclareMathOperator{\End}{End}
\DeclareMathOperator{\Ext}{Ext}
\DeclareMathOperator{\Hom}{Hom}

\let\Im\relax\DeclareMathOperator{\Im}{Im}
\DeclareMathOperator{\Ker}{Ker}
\let\mod\relax\DeclareMathOperator{\mod}{mod}
\DeclareMathOperator{\sub}{sub}

\newcommand{\1}{\mathds{1}}
\newcommand{\CC}{\mathcal{C}}
\newcommand{\CE}{\mathcal{E}}
\newcommand{\CI}{\mathcal{I}}
\newcommand{\CP}{\mathcal{P}}
\newcommand{\epi}{\mathrm{epi}}

\newcommand{\Extc}{\Ext_\mathcal{C}^1}
\newcommand{\Homc}{\Hom_\mathcal{C}}
\newcommand{\Homco}{\overline\Hom_\mathcal{C}}
\newcommand{\Homcu}{\underline\Hom_\mathcal{C}}
\newcommand{\To}{\longrightarrow}
\newcommand{\op}{\mathrm{op}}
\newcommand{\set}[1]{\left\{#1\right\}}

\title[Auslander's defect formula and a commutative triangle]
  {Auslander's defect formula and a commutative triangle
  in an exact category}

\date{\today}

\author{Pengjie Jiao}
\address{Department of Mathematics,
  China Jiliang University,
  Hangzhou 310018, PR China}
\email{jiaopjie@cjlu.edu.cn}

\begin{document}

\begin{abstract}
  We prove Auslander's defect formula in an exact category, and obtain a commutative triangle involving the Auslander bijections and the generalized Auslander--Reiten duality.
\end{abstract}

\maketitle

\section{Introduction}

Throughout, $k$ denotes a commutative artinian ring. We consider a $k$-linear Hom-finite Krull--Schmidt exact category, which is essentially small.

Recall that Auslander's defect formula appeared as \cite[Theorem~III.4.1]{Auslander1978Functors} for the first time.
Krause \cite{Krause2003short} gave a short proof for the formula in module categories.
A higher analogue of the formula was given as \cite[Lemma~3.2]{Iyama2007Higher};
see also \cite[Theorem~3.8]{JassoKvamme2019introduction}.
For an exact category, the notion of \emph{generalized Auslander--Reiten duality} was introduced in \cite[Section~3]{Jiao2018generalized}.
As another analogue, we deduce Auslander's defect formula in an exact category; see Proposition~\ref{prop:defect} below.

The notion of \emph{Auslander bijection} in module categories was introduced by Ringel \cite{Ringel2013Auslander}.
Chen \cite[Theorem~4.6]{Chen2017Auslander} established a commutative triangle involving Auslander bijections, universal extensions and the Auslander--Reiten duality.
Considering the generalized Auslander--Reiten daulity, the commutative triangle could be established locally in an exact category; see Theorem~\ref{thm:triangle} below.

This paper is organized as follows.
In Section~2, we recall the defects and the generalized Auslander--Reiten duality, and deduce Proposition~\ref{prop:defect}.
Section~3 is dedicated to the proof of Theorem~\ref{thm:triangle}.

\section{Auslander's defect formula}

Let $k$ be a commutative artinian ring and let $\check{k}$ be the minimal injective cogenerator.
We denote by $\mod k$ the category of finitely generated $k$-modules and by $D=\Hom_k(-,\check{k})$ the Matlis duality.

\subsection{Exact categories}

Recall that an \emph{exact category} is an additive category $\CC$ together with a collection $\CE$ of exact pairs $(i,d)$, which satisfies the axioms in \cite[Appendix~A]{Keller1990Chain}.
Here, an exact pair $(i,d)$ means a sequence of morphisms $X \xrightarrow{i} Y \xrightarrow{d} Z$ such that $i$ is the kernel of $d$ and $d$ is the cokernel of $i$.
An exact pair $(i,d)$ in $\CE$ is called a \emph{conflation}, while $i$ is called an \emph{inflation} and $d$ is called a \emph{deflation}.
For a pair of objects $X$ and $Y$ in $\CC$, we denote by $\Extc(X,Y)$ the set of equivalence classes of conflations $Y\to E\to X$.

We will consider a $k$-linear Hom-finite Krull--Schmidt exact category $\CC$, which is essentially small. 

Recall from \cite[Section~2]{LenzingZuazua2004Auslander} that a morphism $f\colon X\to Y$ is called \emph{projectively trivial} if for each object $Z$, the induced map
\[
  \Extc(f,Z) \colon \Extc(Y,Z) \To \Extc(X,Z)
\]
is zero.
We observe that $f$ is projectively trivial if and only if $f$ factors through any deflation ending at $Y$.
Dually, we call $f$ \emph{injectively trivial} if for each object $Z$, the induced map
\[
  \Extc(Z,f) \colon \Extc(Z,X) \To \Extc(Z,Y)
\]
is zero.

Given a pair of objects $X$ and $Y$, we denote by $\CP(X,Y)$ the set of projectively trivial morphisms $X \to Y$. Then $\CP$ forms an ideal of $\CC$. We set $\underline{\CC} = \CC/\CP$.
Given a morphism $f \colon X \to Y$, we denote by $\underline{f}$ its image in $\underline{\CC}$.
We denote by
\[
  \Homcu(X,Y) = \Homc(X,Y)/\CP(X,Y)
\]
the set of morphisms $X \to Y$ in $\underline{\CC}$.

Dually, we denote by $\CI(X,Y)$ the set of injectively trivial morphisms $X \to Y$. Set $\overline{\CC} = \CC/\CI$. Given a morphism $f \colon X \to Y$, we denote by $\overline{f}$ its image in $\overline{\CC}$.
We denote by
\[
  \Homco(X,Y) = \Homc(X,Y)/\CI(X,Y)
\]
the set of morphisms $X \to Y$ in $\overline{\CC}$.

Let $K \in \CC$, and let
\[
  \xi \colon X \xrightarrow{i} E \xrightarrow{d} Y
\]
be a conflation.
For any $f\colon K\to Y$, we let $\xi.f$ be the conflation obtained by pullback of $\xi$ along $f$; for any $g\colon X\to K$, we let $g.\xi$ be the conflation obtained by pushout of $\xi$ along $g$.
Consider the \emph{connecting maps}
\begin{align*}
  c(K, \xi) \colon & \Homcu(K, Y) \To \Extc(K, X), \quad
  \underline{f} \mapsto \xi.f,\\
  c(\xi, K) \colon & \Homco(X, K) \To \Extc(Y, K), \quad
  \overline{g}  \mapsto g.\xi.
\end{align*}
They are well defined, since $\xi.f$ splits if $f \in \CP(K,Y)$ and $g.\xi$ splits if $g \in \CI(X,K)$.
We mention the following well-known fact.
\begin{lemma}\label{lem:exact}
  There exist the exact sequences of $k$-modules
  \begin{gather*}
    \Homcu(K,E) \To \Homcu(K,Y) \xrightarrow{c(K,\xi)} \Extc(K,X) \To \Extc(K,E),\\
    \Homco(E,K) \To \Homco(X,K) \xrightarrow{c(\xi,K)} \Extc(Y,K) \To \Extc(E,K),
  \end{gather*}
  natural in both $\xi$ and $K$.
\end{lemma}

\begin{proof}
  We verify the exactness of the first sequence.
  Consider the exact sequence
  \[
    \Homc(K,E) \To \Homc(K,Y) \overset{h}\To \Extc(K,X) \To \Extc(K,E).
  \]
  Here, $h(f) = \xi.f$.
  We observe that $\CP(K,Y) \subseteq \Ker h$ and $c(K,\xi)$ is precisely induced by $h$.
  Then
  \[
    \begin{split}
      \Im \Homcu(K,d) & = \Im \Homc(K,d) + \CP(K,Y) \\
                      & = \Ker h + \CP(K,Y) \\
                      & = \Ker c(K,\xi),
    \end{split}
  \]
  and
  \[
    \Im c(K,\xi) = \Im h = \Ker \Extc(K,i).
    \qedhere
  \]
\end{proof}

\subsection{Generalized Auslander--Reiten duality}

Recall from \cite[Section~3]{Jiao2018generalized} that the \emph{generalized Auslander--Reiten duality} on $\CC$ is a sextuple $\set{\CC_r, \CC_l, \phi, \psi, \tau, \tau^-}$.
Here,
\[
  \CC_r =
  \set{
    X\in\CC \middle|
    D\Extc(X,-) \colon \overline{\CC} \to \mod k
    \mbox{ is representable}
  }
\]
and
\[
  \CC_l =
  \set{
    X\in\CC \middle|
    D\Extc(-,X) \colon \underline{\CC} \to \mod k
    \mbox{ is representable}
  }
\]
are subcategories of $\CC$.
We have a pair of mutually quasi-inverse equivalences
\[
  \tau \colon \underline{\CC_r}
  \overset\sim\To
  \overline{\CC_l}
  \quad\mbox{and}\quad
  \tau^- \colon \overline{\CC_l}
  \overset\sim\To
  \underline{\CC_r}.
\]
For any $X \in \CC$ and $Y \in \CC_r$, we have a natural isomorphism
\[
  \phi_{Y,X} \colon \Homco(X, \tau Y)  \overset\sim\To D\Extc(Y, X).
\]
For any $X \in \CC_l$ and $Y \in \CC$, we have a natural isomorphism
\[
  \psi_{X,Y} \colon \Homcu(\tau^-X, Y) \overset\sim\To D\Extc(Y, X).
\]

For any $Y\in\CC_r$, we let
\[
  \alpha_Y = \phi_{Y,\tau Y} (\overline{\1_{\tau Y}})
  \in D\Extc(Y, \tau Y)
\]
and
\[
  \underline{\mu_Y} =
  \psi_{\tau Y,Y}^{-1} (\alpha_Y)
  \colon \tau^-\tau Y \To Y.
\]
For any $X\in\CC_l$, we let
\[
  \beta_X = \psi_{X,\tau^-X} (\underline{\1_{\tau^-X}})
  \in D\Extc(\tau^-X, X)
\]
and
\[
  \overline{\nu_X} =
  \phi_{\tau^-X,X}^{-1} (\beta_X)
  \colon X \To \tau\tau^-X.
\]

We mention the following relationship between $\alpha$ and $\beta$.

\begin{lemma}\label{lem:linear}
  Let $X \in \CC_l$ and $Y \in \CC_r$. Then
  \[
    \alpha_Y = \beta_{\tau Y} \circ \Extc(\underline{\mu_Y}, \tau Y)
  \]
  and
  \[
    \beta_X = \alpha_{\tau^-X} \circ \Extc(\tau^-X, \overline{\nu_X}).
  \]
\end{lemma}

\begin{proof}
  We only prove the first equality.
  Consider the commutative diagram
  \[\begin{tikzcd}[sep=large, column sep=4em]
    \Homcu(\tau^-\tau Y,\tau^-\tau Y)
      \rar["{\psi_{\tau Y,\tau^-\tau Y}}"]
      \dar["{\Homcu(\tau^-\tau Y,\underline{\mu_Y})}"']
    &D\Extc(\tau^-\tau Y,\tau Y)
      \dar["{D\Extc(\underline{\mu_Y}, \tau Y)}"]
    \\
    \Homcu(\tau^-\tau Y,Y)
      \rar["{\psi_{\tau Y,Y}}"]
    &D\Extc(Y,\tau Y).
  \end{tikzcd}\]
  By a diagram chasing, we obtain
  \[\begin{split}
    \alpha_Y &= \psi_{\tau Y,Y} (\underline{\mu_Y})\\
    &= (\psi_{\tau Y,Y} \circ
      \Homcu(\tau^-\tau Y, \underline{\mu_Y}))
      (\underline{\1_{\tau^-\tau Y}})\\
    &= (D\Extc(\underline{\mu_Y}, \tau Y) \circ \psi_{\tau Y,\tau^-\tau Y})
      (\underline{\1_{\tau^-\tau Y}})\\
    &= D\Extc(\underline{\mu_Y}, \tau Y) (\beta_{\tau Y})\\
    &= \beta_{\tau Y} \circ \Extc(\underline{\mu_Y}, \tau Y).
  \end{split}\]
  Here, the first and the fourth equalities hold by the definitions of $\underline{\mu_Y}$ and $\beta_{\tau Y}$, respectively.
  The third one holds by the commutative diagram.
\end{proof}

We mention the following commutative diagrams;
see \cite[Lemma~4.3]{Chen2017Auslander}.
Here we make the proof more explicit.
It will be used in the proof of Theorem~\ref{thm:triangle} below.

\begin{lemma}
  \label{lem:connect}
  Let $\xi \colon X \to E \to Y$ be a conflation in $\CC$.
  \begin{enumerate}
    \item\label{lem:connect:1}
      For any $K \in \CC_r$, there exists a commutative diagram
      \[\begin{tikzcd}[sep=large]
        \Homco(X, \tau K) &\Extc(Y, \tau K)\\
        D \Extc(K, X)     &D \Homcu(K, Y),
        \ar[from=1-1, to=1-2, "{c(\xi, \tau K)}"]
        \ar[from=1-1, to=2-1, "\phi_{K, X}"']
        \ar[from=2-1, "{D c(K, \xi)}"]
        \ar[from=1-2, "{D (\psi_{\tau K, Y} \circ
          \Homcu(\underline{\mu_K}, Y))}"]
      \end{tikzcd}\]
      which is natural in both $\xi$ and $K$.
    \item\label{lem:connect:2}
      For any $K \in \CC_l$, there exists a commutative diagram
      \[\begin{tikzcd}[sep=large]
        \Homcu(\tau^-K, Y) &\Extc(\tau^-K, X)\\
        D \Extc(Y, K)      &D \Homco(X, K),
        \ar[from=1-1, to=1-2, "{c(\tau^-K, \xi)}"]
        \ar[from=1-1, to=2-1, "\psi_{K, Y}"']
        \ar[from=2-1, "{D c(\xi, K)}"]
        \ar[from=1-2, "{D (\phi_{\tau^-K, X} \circ
          \Homco(X, \overline{\nu_K}))}"]
      \end{tikzcd}\]
      which is natural in both $\xi$ and $K$.
  \end{enumerate}
\end{lemma}

\begin{proof}
  We only prove (\ref{lem:connect:1}). Set
  \[
    \psi' = \psi_{\tau K, Y} \circ \Homcu(\underline{\mu_K}, Y)
    \colon \Homcu(K, Y) \To D \Extc(Y, \tau K).
  \]
  Then we obtain the $k$-linear map
  \[
    D\psi' \colon \Extc(Y, \tau K) \To D\Homcu(K, Y), \quad
    \zeta \mapsto \Big( \underline{f} \mapsto \psi'(\underline{f})(\zeta) \Big).
  \]
  For any $g \colon X \to \tau K$ and $h \colon K \to Y$, we have
  \[\begin{split}
    ( D \psi' \circ c(\xi, \tau K) ) (\overline{g}) (\underline{h})
    &= ( \psi' (\underline{h}) ) ( c(\xi, \tau K) (\overline{g}) )\\
    &= ( \psi' (\underline{h}) ) (g.\xi)\\
    &= \psi_{\tau K, Y} (\underline{h \circ \mu_K}) (g.\xi).
  \end{split}\]
  Here, the second and the third equalities hold by the definitions of $c(\xi, \tau K)$ and $\psi'$, respectively.

  On the other hand, by the definition of $c(K, \xi)$, we obtain
  \[
    (D c(K, \xi) \circ \phi_{K, X}) (\overline{g}) (\underline{h})
    = ( \phi_{K, X} (\overline{g}) \circ c(K, \xi) ) (\underline{h})
    = \phi_{K, X} (\overline{g}) (\xi.h).
  \]
  Consider the following commutative diagram
  \[\begin{tikzcd}[sep=large]
    \Homco(\tau K,\tau K)
      \rar["{\phi_{K, \tau K}}"]
      \dar["{\Homco(\overline{g}, \tau K)}"']
    &D\Extc(K, \tau K)
      \dar["{D\Extc(K, \overline{g})}"]
    \\
    \Homco(X, \tau K)
      \rar["{\phi_{K, X}}"]
    &D\Extc(K, X).
  \end{tikzcd}\]
  We have
  \[\begin{split}
    \phi_{K, X} (\overline{g})
    &= ( \phi_{K, X} \circ \Homco(\overline{g}, \tau K) )
       (\overline{\1_{\tau K}})\\
    &= ( D\Extc(K, \overline{g}) \circ \phi_{K, \tau K} )
       (\overline{\1_{\tau K}})\\
    &= \alpha_K \circ \Extc(K, \overline{g})\\
    &= \beta_{\tau K} \circ \Extc(\underline{\mu_K}, \tau K) \circ \Extc(K, \overline{g}).
  \end{split}\]
  Here, the third equality holds by the definition of $\alpha_K$. The fourth one holds by Lemma~\ref{lem:linear}.
  It follows that
  \[\begin{split}
    (D c(K, \xi) \circ \phi_{K, X}) (\overline{g}) (\underline{h})
    &= ( \beta_{\tau K} \circ \Extc(\mu_K, \tau K) \circ \Extc(K, \overline{g}) )
       (\xi.h)\\
    &= \beta_{\tau K} (g.\xi.h.\mu_K)\\
    &= ( \beta_{\tau K} \circ \Extc(\underline{h \circ \mu_K}, \tau K) ) (g.\xi).
  \end{split}\]
  Then it remains to show
  \[
    \psi_{\tau K, Y} (\underline{h \circ \mu_K}) =
    \beta_{\tau K} \circ \Extc(\underline{h \circ \mu_K}, \tau K).
  \]

  Consider the following commutative diagram
  \[\begin{tikzcd}[sep=large, column sep=4em]
    \Homcu(\tau^-\tau K, \tau^-\tau K)
      \rar["{\psi_{\tau K, \tau^-\tau K}}"]
      \dar["{\Homcu(\tau^-\tau K, \underline{h \circ \mu_K})}"']
    &D\Extc(\tau^-\tau K, \tau K)
      \dar["{D\Extc(\underline{h \circ \mu_K}, \tau K)}"]
    \\
    \Homcu(\tau^-\tau K, Y)
      \rar["{\psi_{\tau K, Y}}"]
    &D\Extc(K, \tau Y).
  \end{tikzcd}\]
  We have
  \[\begin{split}
    \psi_{\tau K, Y} (\underline{h \circ \mu_K})
    &= ( \psi_{\tau K, Y} \circ \Homcu(\tau^-\tau K, \underline{h \circ \mu_K}) )
       (\underline{\1_{\tau^-\tau K}})\\
    &= ( D\Extc(\underline{h \circ \mu_K}, \tau K) \circ \psi_{\tau K, \tau^-\tau K} )
       (\underline{\1_{\tau^-\tau K}})\\
    &= \beta_{\tau K} \circ \Extc(\underline{h \circ \mu_K}, \tau K).
  \end{split}\]
  Here, the third equality holds by the definition of $\beta_{\tau Y}$.
  Then the result follows.
\end{proof}

\subsection{Defect formula}

Let $\xi \colon X \xrightarrow{i} E \xrightarrow{d} Y$ be a conflation.
The \emph{covariant defect} $\xi_*$ and the \emph{contravariant defect} $\xi^*$ are defined by the exact sequences of functors
\begin{gather*}
  0 \To \Homc(Y,-) \To \Homc(E,-) \To \Homc(X,-) \To \xi_* \To 0,\\
  0 \To \Homc(-,X) \To \Homc(-,E) \To \Homc(-,Y) \To \xi^* \To 0.
\end{gather*}

Consider the following exact sequences of functors
\begin{gather*}
  \Homc(E,-) \To \Homc(X,-) \To \Extc(Y,-) \To \Extc(E,-),\\
  \Homc(-,E) \To \Homc(-,Y) \To \Extc(-,X) \To \Extc(-,E).
\end{gather*}
It follows that
\[
  \xi_* \cong \Ker \Extc(d,-)
  \quad \mbox{and} \quad
  \xi^* \cong \Ker \Extc(-,i).
\]
In particular, $\xi_*$ vanishes on injectively trivial morphisms and $\xi^*$ vanishes on projectively trivial morphisms.
They induce the functors
\[
  \xi_* \colon \overline{\CC} \to \mod k
  \quad \mbox{and} \quad
  \xi^* \colon \underline{\CC} \to \mod k.
\]
Moreover, following Lemma~\ref{lem:exact}, they can be expressed as follows.
\begin{lemma}\label{lem:defect}
  There exist isomorphisms of functors
  \[
    \begin{aligned}[b]
      \Ker \Extc(d, -) \cong \xi_* \cong \Cok \Homco(i, -)
      \colon & \overline{\CC} \To \mod k,\\
      \Ker \Extc(-, i) \cong \xi^* \cong \Cok \Homcu(-, d)
      \colon & \underline{\CC} \To \mod k.
    \end{aligned}
    \eqno\qed
  \]
\end{lemma}

Following \cite{Krause2003short}, Auslander's defect formula is a consequence of the previous fact and the generalized Auslander--Reiten duality;
compare \cite[Theorem~3.8]{JassoKvamme2019introduction}
and \cite[Theorem~III.4.1]{Auslander1978Functors}.

\begin{proposition}\label{prop:defect}
  Let $\xi \colon X \xrightarrow{i} E \xrightarrow{d} Y$ be a conflation in $\CC$.
  \begin{enumerate}
    \item\label{prop:defect:1}
      For any $K \in \CC_r$, there exists an isomorphism
      \[
        \xi_* (\tau K) \cong D \xi^* (K),
      \]
      which is natural in both $\xi$ and $K$.
    \item\label{prop:defect:2}
      For any $K \in \CC_l$, there exists an isomorphism
      \[
        \xi^* (\tau^-K) \cong D \xi_* (K),
      \]
      which is natural in both $\xi$ and $K$.
  \end{enumerate}
\end{proposition}

\begin{proof}
  We only prove (\ref{prop:defect:1}).
  For any $K \in \CC_r$, we have the natural isomorphisms
  \[
    \begin{split}
      D \xi^*(K) & \cong D \Cok \Homcu(K, d) \\
                 & \cong \Ker D \Homcu(K, d) \\
                 & \cong \Ker \Extc(d, \tau K) \\
                 & \cong \xi_*(\tau K).
    \end{split}
  \]
  Here, the first and the last isomorphisms follow from Lemma~\ref{lem:defect}.
  The third one follows from the generalized Auslander--Reiten duality.
  Their naturality in $\xi$ is a direct verification.
%
%
\end{proof}

\section{A commutative triangle}

Recall that two morphism $f \colon X \to Y$ and $f' \colon X' \to Y$ are called \emph{right equivalent} if they factor through each other.
Denote by $[f\rangle$ the right equivalence class containing $f$, and by $[\to Y \rangle$ the class of right equivalence classes of morphisms ending to $Y$.
We mention that $[\to Y \rangle$ is a poset; see \cite[Section~I.2]{Ringel2013Auslander}.
Here, $[ f \rangle \leq [f' \rangle$ means that $f$ factors through $f'$. We denote by $[\to Y \rangle_\epi$ the subset of $[\to Y \rangle$ formed by deflations.

Let $C \in \CC$ and $\Gamma(C) = \End_\CC(C)$. Denote by $\add C$ the category of direct summands of finite direct sums of $C$.
For a $\Gamma(C)$-module $M$, we denote by $\sub_{\Gamma(C)} M$ the lattice of its submodules.
For any $Y \in \CC$, we have a morphism of posets
\[
  \eta_{C,Y} \colon [\to Y \rangle_\epi \To
  \sub_{\Gamma(C)^\op} \Homcu(C, Y), \quad
  [d\rangle \mapsto \Im \Homcu(C, d).
\]

Recall that a morphism $f \colon X \to Y$ is called \emph{right $C$-determined}, if for any $T \in \CC$ and any $g \colon T \to Y$, the following conditions are equivalent.
\begin{enumerate}
  \item $g$ factors through $f$.
  \item $g \circ h$ factors through $f$ for any $h \colon C \to T$.
\end{enumerate}

Denote by $^C[\to Y \rangle_\epi$ the subset of $[\to Y \rangle_\epi$ consisting of classes $[d \rangle$ that $d$ is a right $C$-determined deflation, and by $_C[\to Y \rangle_\epi$ the one consisting of classes $[d \rangle$ that there exists some deflation $d' \colon X' \to Y$ with $d' \in [d\rangle$ and $\Ker d \in \add C$.

We mention the following fact.
Here, we follow the term of \cite[Proposition~4.5]{Chen2017Auslander};
see also \cite[Section~XI.2]{AuslanderReitenSmalo1995Representation}
and \cite{Ringel2012Morphisms}.

\begin{lemma}
  \label{lem:epi-det}
  If $K \in \CC_l$, then
  \[
    _K[\to Y \rangle_\epi = {^{\tau^-K}[\to Y \rangle_\epi}.
  \]
\end{lemma}

\begin{proof}
  We observe that each deflation whose kernel lies in $\add K$ is right $\tau^-K$-determined; see \cite[Lemma~4.2]{JiaoLe2018Auslander}.
  It follows that
  \[
    _K[\to Y \rangle_\epi \subseteq {^{\tau^-K}[\to Y \rangle_\epi}.
  \]

  On the other hand, let $d \colon X \to Y$ be a right $\tau^-K$-determined deflation.
  One can see that $\Im \Homc(\tau^-K, d)$ is a $\Gamma(\tau^-K)^\op$-submodule of $\Homc(\tau^-K, Y)$.
  Since $d$ is a deflation, we have $\CP(\tau^-K, Y) \subseteq \Im \Homc(\tau^-K, d)$.
  We observe that there exists some right $\tau^-K$-determined deflation $d'$ such that $\Ker d' \in \add K$ and $\Im \Homc(\tau^-K, d') = \Im \Homc(\tau^-K, d)$;
  see \cite[Theorem~4.3]{JiaoLe2018Auslander}.
  It is a direct verification that $d$ and $d'$ are right equivalent, since they are both right $\tau^-K$-determined.
  Then $[d\rangle$ lies in $_K[\to Y \rangle_\epi$.
  It follows that
  \[
    _K[\to Y \rangle_\epi \supseteq {^{\tau^-K}[\to Y \rangle_\epi}.
    \qedhere
  \]
\end{proof}

Given a deflation $d \colon X \to Y$, we denote by $\xi_d$ the conflation $\Ker d \to X \xrightarrow{d} Y$.
For any $K \in \CC$, we set
\[
  \delta_{K,Y} ([d\rangle) = \Im c(\xi_d, K),
\]
which is a $\Gamma(K)$-submodule of $\Extc(Y, K)$.
We mention the following fact.

\begin{lemma}[{\cite[Propositions~2.4]{Chen2017Auslander}}]
  \label{lem:delta}
  For any $K \in \CC$, the restricted map
  \[
    \delta_{K,Y} \colon {_K[\to Y \rangle_\epi} \To \sub_{\Gamma(K)} \Extc(Y, K)
  \]
  is an anti-isomorphism of posets.
  \qed
\end{lemma}

Consider the natural isomorphism
\[
  \psi_{K,Y} \colon \Homcu(\tau^-K, Y) \To D \Extc(Y, K).
\]
It induces an anti-isomorphism of posets
\[\begin{split}
  \gamma_{K, Y} \colon \sub_{\Gamma(K)} \Extc(Y, K)
  &\To \sub_{\Gamma(\tau^-K)^\op} \Homcu(\tau^-K, Y),\\
  L &\longmapsto \psi_{K,Y}^{-1}(D (\Extc(Y, K)/L) ).
\end{split}\]
Here, we identify $D (\Extc(Y, K)/L)$ with the subset of $D \Extc(Y, K)$ formed by the $k$-linear maps vanishing on $L$.

The following result establishes the Auslander bijection. This is a slight generalization of \cite[Theorem~4.6]{Chen2017Auslander}.

\begin{theorem}\label{thm:triangle}
  For any $Y \in \CC$ and $K \in \CC_l$, there exists the commutative bijection triangle
  \[\begin{tikzcd}[column sep=0em]
    &\sub_{\Gamma(K)} \Extc(Y, K)\\
    ^{\tau^-K}[\to Y \rangle_\epi &&\sub_{\Gamma(\tau^-K)^\op} \Homcu(\tau^-K, Y).
    \ar[from=2-1, to=1-2, start anchor=north east, "{\delta_{K, Y}}"]
    \ar[from=2-1, "{\eta_{\tau^-K, Y}}"]
    \ar[from=1-2, end anchor=north west, "{\gamma_{K, Y}}"]
  \end{tikzcd}\]
\end{theorem}

\begin{proof}
  It is sufficient to show that the triangle is commutative.
  Given any deflation $d \colon X \to Y$, Lemma~\ref{lem:exact} gives the exact sequence of $k$-modules
  \[
    \Homcu(\tau^-K, X) \xrightarrow{\Homcu(\tau^-K, d)}
    \Homcu(\tau^-K, Y) \xrightarrow{c(\tau^-K, \xi_d)}
    \Extc(\tau^-K, \Ker d).
  \]
  Then
  \begin{equation}\label{eq:eta(d)}
    \eta_{\tau^-K, Y} ([d\rangle)
    = \Im \Homcu(\tau^-K, d)
    = \Ker c(\tau^-K, \xi_d).
  \end{equation}

  Applying $D$ to the exact sequence of $k$-modules
  \[
    \Homco(X, K) \xrightarrow{c(\xi_d, K)} \Extc(Y, K) \To \Extc(Y, K)/\Im c(\xi_d, K) \To 0,
  \]
  we obtain
  \[
    D(\Extc(Y, K)/\Im c(\xi_d, K)) = \Ker D c(\xi_d, K).
  \]
  By the commutative diagram in Lemma~\ref{lem:connect}(\ref{lem:connect:2}), we have
  \begin{equation}\label{eq:psi(Ker)}
    \psi_{K, Y} (\Ker c(\tau^-K, \xi_d))
    = \Ker D c(\xi_d, K)
    = D(\Extc(Y, K)/\Im c(\xi_d, K)).
  \end{equation}
  Then
  \[\begin{split}
    (\gamma_{K, Y} \circ \delta_{K, Y}) ([d\rangle)
    &= \psi_{K, Y}^{-1} (D(\Extc(Y, K)/\Im c(\xi_d, K)))\\
    &= \Ker c(\tau^-K, \xi_d)\\
    &= \eta_{\tau^-K, Y} ([d\rangle).
  \end{split}\]
  Here, the first equality holds by the definitions of $\delta_{K, Y}$ and $\gamma_{K, Y}$.
  The second and the third ones are just (\ref{eq:psi(Ker)}) and (\ref{eq:eta(d)}), respectively.
\end{proof}

\section*{Acknowledgements}

The author thanks Professor~Xiao-Wu Chen for his encouragement and suggestions.
He also thanks the referees for many helpful suggestions and comments.
This work was supported by the National Natural Science Foundation of China (Grant No.~11901545).


\end{document}